\documentclass[10pt, reqno]{amsart}
\usepackage[utf8]{inputenc}

\usepackage{graphicx}
\usepackage{fourier}
\usepackage[T1]{fontenc}
\usepackage[margin=1in]{geometry}
\usepackage{parskip}
\usepackage{amsthm}
\usepackage{amsmath}
\usepackage{amssymb}
\usepackage{indentfirst}

\usepackage[
backend=biber,
style=numeric,
sorting=nyt,
maxnames=10,
giveninits=true
]{biblatex}
\newcommand {\Z}{\mathbb{Z}}

\newcommand{\Mod}[1]{\ (\mathrm{mod}\ #1)}
\newtheoremstyle{exampstyle}
  {20pt} % Space above
  {\topsep} % Space below
  {} % Body font
  {} % Indent amount
  {\bfseries} % Theorem head font
  {.} % Punctuation after theorem head
  {.5em} % Space after theorem head
  {} % Theorem head spec (can be left empty, meaning `normal')
\theoremstyle{exampstyle}
\newtheorem{thm}{Theorem}
\newtheorem*{defn}{Definition}
\newtheorem*{assump}{Assumption ($\ast$)}
\newtheorem{lemma}[thm]{Lemma}

\newtheorem{cor}{Corollary}[thm]
\newtheorem*{remark}{Remark}
\usepackage{multicol}
\usepackage{xcolor}
\usepackage{hyperref}
\hypersetup{
    colorlinks,
    linkcolor={red!50!black},
    citecolor={blue!50!black},
    urlcolor={blue!80!black}
}

\title{On the weakly prime-additive numbers with length 4}
\author{Wing Hong Leung}
\date{}

\begin{document}

\maketitle

\noindent \textbf{Abstract}: In 1992, Erd\H{o}s and Hegyv\'{a}ri showed that for any prime $p$, there exist infinitely many length 3 weakly prime-additive numbers divisible by $p$. In 2018, Fang and Chen showed that for any positive integer $m$, there exist infinitely many length 3 weakly prime-additive numbers divisible by $m$ if and only if $8$ does not divide $m$. Assuming the existence of a prime in certain arithmetic progressions with prescribed primitive root, which is true under the Generalized Riemann Hypothesis (GRH), we show that for any positive integer $m$, there exist infinitely many length 4 weakly prime-additive numbers divisible by $m$. We also present another related result analogous to the length 3 case shown by Fang and Chen.

\section{Introduction}

A number $n$ with at least 2 distinct prime divisors is called \textit{prime-additive} if $n=\sum_{p|n}p^{a_p}$ for some $a_p>0$. If additionally, $p^{a_p}<n\leq p^{a_p+1}$ for all $p|n$, then $n$ is called \textit{strongly prime-additive}. In 1992, Erd\H{o}s and Hegyv\'{a}ri \cite{erdos} stated a few examples and conjectured that there are infinitely many strongly prime-additive numbers. However, this problem was and is still far from being solved. For example, not even the infinitude of prime-additive numbers is known. Therefore they introduced the following weaker version of prime-additive numbers.

\begin{defn}
A positive integer $n$ is said to be \textit{weakly prime-additive} if $n$ has at least 2 distinct prime divisors, and there exists distinct prime divisors $p_1,...,p_t$ of $n$ and positive integers $a_1,...,a_t$ such that $n=p_1^{a_1}+\cdots+p_t^{a_t}$. The minimal value of such $t$ is defined to be the \textit{length} of $n$, denoted as $\kappa_n$.
\end{defn}

Note that if $n$ is a weakly prime-additive number, then $\kappa_n\geq3$. So we call a weakly prime-additive number with length 3 a \textit{shortest weakly prime-additive} number.

Erd\H{o}s and Hegyv\'{a}ri \cite{erdos} showed that for any prime $p$, there exist infinitely many weakly prime-additive numbers divisible by $p$. In fact, they showed that these weakly prime-additive numbers can be taken to be shortest weakly prime-additive in their proof. They also showed that the number of shortest weakly prime-additive numbers up to some integer $N$ is at least $c(\log N)^3$ for a sufficiently small constant $c>0$.

In 2018, Fang and Chen \cite{fang2018shortest} showed that for any positive integer $m$, there exist infinitely many shortest weakly prime-additive numbers divisible by $m$ if any only if $8$ does not divide $m$. This is Theorem \ref{shortest} stated in this paper. They also showed that for any positive integer $m$, there exist infinitely many weakly prime-additive numbers with length $\kappa_n\leq 5$ that are divisible by $m$. In the same paper, Fang and Chen posted four open problems. The first one inquires whether, for any positive integer $m$, there are infinitely many weakly prime-additive numbers $n$ with $m|n$ and $\kappa_n=4$. In Theorem \ref{main} of this paper, we confirm this is true, assuming the existence of a prime in certain arithmetic progressions with prescribed primitive root (see assumption \hyperref[assumption]{($\ast$)} on p.2). This assumption is known to hold under the Generalized Riemann Hypothesis (GRH).

Finally, it was also shown in \cite{fang2018shortest} that for any distinct primes $p,q$, there exists a prime $r$ and infinitely many $a,b,c$ such that $pqr|p^a+q^b+r^c$. In Theorem \ref{main2}, we extend this result analogously to four distinct primes, subject to mild congruence conditions, assuming the same assumption as mentioned above.

\section{Main Results}

\begin{assump}\label{assumption}
Let $1\leq a\leq f$ be positive integers with $(a,f)=1$ and $4|f$. Let $g$ be an odd prime dividing $f$ such that $\left(\frac{g}{a}\right)=-1$ with $\left(\frac{\cdot}{\cdot}\right)$ being the Kronecker symbol. Then there exists a prime $p$ such that $p\equiv a\Mod f$ and $g$ is a primitive root of $p$.
\end{assump}

It is known that \hyperref[assumption]{($\ast$)} is a consequence of the Generalized Riemann Hypothesis (GRH), see Corollary \ref{moree} in the next section for details. Under the assumption \hyperref[assumption]{($\ast$)}, we have the following.

\begin{thm}\label{main}
Assume \hyperref[assumption]{($\ast$)}. For any positive integer $m$, there exist infinitely many weakly prime-additive numbers $n$ with $m|n$ and $\kappa_n=4$.
\end{thm}

Note that if a positive integer $n$ can be expressed in the form of $n=p^a+q^b+r^c+s^d$ for some distinct primes $p,q,r,s$, and positive integers $a,b,c,d$ such that $p,q,r,s|n$, then $p,q,r,s$ are all odd primes. We have the following theorem as a partial converse.

\begin{thm}\label{main2}
Assume \hyperref[assumption]{($\ast$)}. For any distinct odd primes $p,q,r$ with one of them $\equiv 3$ or $5\Mod 8$, there exist infinitely many prime $s$ and infinitely many positive integers $a,b,c,d$ such that $$pqrs|p^a+q^b+r^c+s^d.$$
\end{thm}

This is analogous to Theorem 1.4 in \cite{fang2018shortest}, which says that for any given distinct primes $p, q$, there exists a prime $r>\max\{p, q\}$ and infinitely many triples $(a, b, c)$ of positive integers such that $pqr|p^a+q^b+r^c$.

\section{Preliminaries}

\begin{lemma}[{\normalfont\cite[Thm 72]{hardy1979introduction}}]\label{fermat}
{\,\normalfont(The Fermat-Euler Theorem)} Let $a,n$ be coprime positive integers. Then $$a^{\phi(n)}\equiv1\Mod n,$$
where $\phi$ is the Euler totient function.
\end{lemma}

We will use the Kronecker Symbol $(\frac{\cdot}{\cdot})$, which is a generalization of the Legendre symbol. Precisely, this is defined as follows. Let $a,b$ be integers. If $b=0$ or $b=\pm1$, we define \begin{align*}
    \left(\frac{a}{0}\right)=\begin{cases}
        1& \text{ if } a =\pm1\\ 0 &\text{ otherwise}
    \end{cases} \quad \text{ and } \quad \left(\frac{a}{\pm1}\right)= \begin{cases}\pm 1 & \text { if } a<0 \\ 1 & \text { if } a \geq 0.\end{cases}
\end{align*}
For the remaining cases, let $b=\pm p_1^{e_1}\cdots p_k^{e_k}$ be the prime factorization of $b$. We then define \begin{align*}
    \left(\frac{a}{b}\right)=\left(\frac{a}{\pm1}\right)\prod_{i=1}^k\left(\frac{a}{p_k}\right)^{e_k},
\end{align*}
where for any prime $p$, \begin{align*}
    \left(\frac{a}{p}\right)= \begin{cases}1 & \text { if } a \text { is a quadratic residue modulo } p \text { and } a \neq 0\Mod {p} \\ -1 & \text { if } a \text { is a quadratic nonresidue modulo } p \\ 0 & \text { if } a \equiv 0\Mod{p}\end{cases}
\end{align*}
is the Legendre symbol.

Whenever we write $(\frac{a}{b})$ for some integers $a,b$, it refers to the Kronecker symbol. We will need the following properties of the Kroncecker symbol. See, for example, \cite[p. 289-290]{ayoub1963introduction} for a proof.

\begin{lemma}\label{kronecker}
Let $a,b,c$ be any nonzero integers, and $p,q$ be any odd primes. Let $a'$, $b'$ be the odd part of $a$ and $b$, respectively. Then we have:\\
{\allowdisplaybreaks
\begin{align*}
1. & \left(\frac{ab}{c}\right)=\left(\frac{a}{c}\right)\left(\frac{b}{c}\right) \text{ unless } c=-1;\\
2. & \left(\frac{a}{b}\right)=(-1)^{\frac{a'-1}{2}\frac{b'-1}{2}}\left(\frac{b}{a}\right);\\
3. & \left(\frac{-2}{p}\right)=\begin{cases}1&\text{ if } p\equiv 1,3\Mod 8\\ -1&\text{ if } p\equiv 5,7\Mod 8;\end{cases}\\
4. & \left(\frac{a}{p}\right)\equiv a^\frac{p-1}{2}\Mod p;\\
5. & \left(\frac{p}{q}\right)=\left(\frac{q}{p}\right) \text{ unless } p\equiv q\equiv 3\Mod 4;\\
6. & \,\,\text{If } p\equiv q\equiv 3\Mod 4, \left(\frac{p}{q}\right)=-\left(\frac{q}{p}\right).
\end{align*}}
\end{lemma}

On primes in arithmetic progressions, we have the celebrated Dirichlet's theorem. See, for example, \cite[Chapter 1]{ayoub1963introduction} for a proof.

\begin{thm}\label{dirichlet}
{\normalfont (Dirichlet's Theorem)} If $a,d$ are coprime positive integers, then there are infinitely many primes $p$ such that $p\equiv a\Mod d$. 
\end{thm}

Under GRH, we have the following generalization.

\begin{thm}[{\normalfont\cite[Thm 1.3]{moree2008primes}}]\label{more}
Let $1\leq a\leq f$ be positive integers with $(a,f)=1$. Let $g$ be an integer that is not equal to $-1$ or a square, and let $h\geq1$ be the largest integer such that $g$ is an $h$th power. Write $g=g_1g_2^2$ with $g_1$ square free, and $g_1,g_2\in\Z$. Let $$\beta=\frac{g_1}{(g_1,f)} \text{ and } \gamma_1=\begin{cases}(-1)^\frac{\beta-1}{2}(f,g_1) & \text{if }\beta \text{ is odd;}\\1 & \text{otherwise.}\end{cases}$$
Let $\pi_g(x;f,a)$ be the number of primes $p\leq x$ such that $p\equiv a\Mod{f}$ and $g$ is a primitive root $\Mod{p}$. Then, assuming GRH, we have $$\pi_g(x;f,a)=\delta(a,f,g)\frac{x}{\log x}+O_{f,g}\left(\frac{x\log\log x}{\log^2 x}\right),$$
where $$\delta(a,f,g)=\frac{A(a,f,h)}{\phi(f)}\left(1-\left(\frac{\gamma_1}{a}\right)\frac{\mu(|\beta|)}{\prod_{\substack{p|\beta\\ p|h}}(p-1)\prod_{\substack{p|\beta \\ p\nmid h}}(p^2-p-1)}\right)$$
if one of the following holds: \begin{itemize}
    \item $g_1\equiv1\Mod{4}$,
    \item $g_1\equiv 2\Mod{4}$ and $8|f$
    \item $g_1\equiv3\Mod{4}$ and $4|f$.
\end{itemize}
Otherwise, we have $$\delta(a,f,g)=\frac{A(a,f,h)}{\phi(f)}.$$ 
Here $\mu$ is the M\"{o}bius function, $\left(\frac{\cdot}{\cdot}\right)$ is the Kronecker symbol, and $$A(a,f,h)=\prod_{p|(a-1,f)}\left(1-\frac{1}{p}\right)\prod_{\substack{p\nmid f \\ p|h}}\left(1-\frac{1}{p-1}\right)\prod_{\substack{p\nmid f\\ p\nmid h}}\left(1-\frac{1}{p(p-1)}\right)$$ if $(a-1,f,h)=1$, and $A(a,f,h)=0$ otherwise.
\end{thm}

\begin{cor}\label{moree}
Assume GRH. Let $a,f,g$ be as above and $\left(\frac{g}{a}\right)=-1$. There exists a prime $p$ such that $p\equiv a\Mod f$ and $g$ is a primitive root of $p$. In other words, assumption \hyperref[assumption]{($\ast$)} holds true under GRH.
\end{cor}
\begin{proof}
This corresponds to a special case of Theorem \ref{more}, with our specific conditions on $a,f,g$, $\beta=h=1$, $\gamma_1=g$. Notice that 
$$\delta(a,f,g)=\frac{2}{\phi(f)}\prod_{p|(a-1,f)}\left(1-\frac{1}{p}\right)\prod_{p\nmid f}\left(1-\frac{1}{p(p-1)}\right)>0.$$
\end{proof}

\begin{remark}
This shows that our result also follows from GRH, which is a much stronger assumption than \hyperref[assumption]{($\ast$)}.
\end{remark}

\begin{thm}[{\normalfont\cite[Cor 1.1]{fang2018shortest}}]\label{shortest}
For any positive integer $m$, there exist infinitely many shortest weakly prime-additive numbers $n$ with $m|n$ if and only if 8 does not divide $m$.
\end{thm}

\section{Proof of Theorem 1}

We first prove the following weaker version of Theorem 1.
\begin{thm}\label{mainweaker}
Assume \hyperref[assumption]{($\ast$)}. For any positive integer $m$, there exist infinitely many weakly prime-additive numbers $n$ with $m|n$ and $\kappa_n\leq4$.
\end{thm}
\begin{proof}
Let $m$ be a positive integer. Write $m=2^km_1$ with $(m_1,2)=1$ and $k\geq0\in\Z$. Without loss of generality, we assume $k\geq3$. In fact, if the theorem holds when $m$ is replaced by $2^{\max\{3,k\}}m$, then the theorem holds for $m$. We will construct a family of distinct primes $p,q,r,s$ and positive integers $a,b,c$ such that each of $m,p,q,r,s$ divides $n$, where $n:=p^a+q^b+r^c+s$.

Let $p$ be an odd prime such that $(p,m)=1$. By the Chinese Remainder Theorem and Theorem \ref{dirichlet}, there exists an odd prime $q$ such that $$q\equiv 1 \Mod{2^kp}\text{ and }q\equiv-1\Mod{m_1}.$$
Similarly, we can use the same two theorems to conclude that there exists an odd prime $r$ such that $$r\equiv 3\Mod{2^k}\text{ and }r\equiv1\Mod{pqm_1}.$$
Applying the Chinese Remainder Theorem again, there exists a unique integer $s_0$ such that $1\leq s_0\leq pqrm$ and \begin{align*}\label{acong}
&s_0\equiv -5\Mod{2^k}\\
&s_0\equiv -1\Mod{m_1}\\
&s_0\equiv -2\Mod{pqr}.
\end{align*}
Note that $(s_0,pqrm)=1$.

Since $k\geq3$, we have $r\equiv3\Mod{2^k}$ and $s_0\equiv -5\Mod{2^k}$. This implies that $r\equiv 3\Mod 8$ and $s_0\equiv 3\Mod 8$. Using Lemma \ref{kronecker}, we observe that \begin{align*}
    \left(\frac{r}{s_0}\right)=\left(\frac{s_0}{r}\right)(-1)^{\frac{s_0-1}{2}\frac{r-1}{2}}=-\left(\frac{s_0}{r}\right)=-\left(\frac{-2}{r}\right)=-1.
\end{align*}
Here we used $s_0\equiv -2\Mod r$ as $s_0\equiv -2\Mod{pqr}$. Therefore, applying Corollary \ref{moree} with $a=s_0$, $f=pqrm$ and $g=r$, there exists an odd prime $s$ such that $s\equiv s_0\Mod{pqrm}$, and $r$ is a primitive root of $s$. Consequently, $s$ satisfies all the previously mentioned congruence relations satisfied by $s_0$. Furthermore, there exists a positive integer $c_0$ such that $$r^{c_0}\equiv -2 \Mod s.$$
Note that by construction, $p,q,r,s$ are all distinct odd primes.

Now for any positive integer $c'$, take $$c=(p-1)(q-1)(r-1)\phi(m)c'+c_0.$$
For any positive odd integer $b'$, take $$b=\frac{1}{4}(r-1)(s-1)b'.$$
Since $r\equiv 3\Mod{2^k}$ and $s\equiv -5\Mod{2^k}$, we have $r,s\equiv 3\Mod 4$, ensuring that $b$ is odd. For any positive integer $a'$, take $$a=(q-1)(r-1)(s-1)\phi(m)a',$$
where $\phi$ is the Euler totient function. Finally, let $$n=p^a+q^b+r^c+s.$$
Note that we have the following congruence conditions:
\begin{enumerate}
    \item As $q\equiv r\equiv1\Mod p$, $s\equiv -2\Mod p$, we have $$n\equiv p^a+q^b+r^c+s\equiv0+1+1-2\equiv 0\Mod p.$$
    \item Since $q-1|a$, Lemma \ref{fermat} implies that $p^a\equiv1\mod q$. Hence we have $$n\equiv p^a+q^b+r^c+s\equiv1+0+1-2\equiv0\Mod q.$$
    \item Similarly, $p^a\equiv1\Mod r$ as $r-1|a$. Since $q\equiv1\Mod{2^kp}$, we have $q\equiv 1\Mod 8$. Applying Lemma \ref{kronecker} with $r\equiv3\Mod 8$ and $r\equiv1\Mod q$, $$q^b\equiv (q^{\frac{1}{2}(r-1)})^{\frac{1}{2}(s-1)b'}\equiv\left(\frac{q}{r}\right)^{\frac{1}{2}(s-1)b'}\equiv\left(\frac{r}{q}\right)\equiv\left(\frac{1}{q}\right)\equiv1\Mod{r}.$$
    So we have $$n\equiv p^a+q^b+r^c+s\equiv1+1+0-2\equiv0\Mod r.$$
    \item Similarly, $p^a\equiv q^b\equiv1\Mod s$. As $r^c\equiv-2\Mod s$, we have $$n\equiv p^a+q^b+r^c+s\equiv1+1-2+0\equiv 0\Mod s.$$
    \item As $\phi(m)|a$, Lemma \ref{fermat} gives us $p^a\equiv1\Mod m$. Since $b$ is odd and $q\equiv-1\Mod{m_1}$, we get $q^b\equiv -1\Mod{m_1}$. Together with $r\equiv1\Mod{m_1}$ and $s\equiv -1\Mod{m_1}$, we have $$n\equiv p^a+q^b+r^c+s\equiv1-1+1-1\equiv0\Mod{m_1}.$$
    \item Since $p^a\equiv1\Mod m$, $q\equiv 1\Mod{2^k}$, $r\equiv 3\Mod{2^k}$ and $s\equiv -5\Mod{2^k}$, we have $$n\equiv p^a+q^b+r^c+s\equiv1+1+3-5\equiv0\Mod{2^k}.$$
\end{enumerate}

As a result, $n=p^a+q^b+r^c+s$ is weakly prime additive and is divisible by $m$. Since $a',c'$ can be any positive integers, $b'$ can be any positive odd integer and $p$ can be any arbitrary odd prime that is coprime to $m$, we have constructed infinitely many weakly prime-additive $n$ with length at most $4$.
\end{proof}

\begin{remark}
In the above construction, $s$ can be raised to any $d$-th power for any positive integer $d\equiv 1\Mod{\phi(pqrm)}$.
\end{remark}

Together with Theorem \ref{shortest}, we can now prove Theorem \ref{main}.

\begin{proof}[Proof of Theorem \ref{main}]
Let $m$ be a positive integer. By Theorem \ref{mainweaker}, there exist infinitely many weakly prime-additive numbers with length $\leq4$ such that they are divisible by $8m$. Since $8|8m$, Theorem \ref{shortest} implies that these numbers cannot be shortest weakly prime-additive, and hence they are all weakly prime-additive numbers with length 4.
\end{proof}

\section{Proof of Theorem \ref{main2}}
Let $p,q,r$ be distinct odd primes, with one of them, WLOG say $r$, satisfying $r\equiv 3$ or $5\Mod 8$. Let $k$ be the positive integer such that $\frac{(p-1)(q-1)}{2^k}$ is odd. We denote this value as $A=\frac{(p-1)(q-1)}{2^k}$ and set $f=8Apqr$. By the Chinese Remainder Theorem, there exists a unique integer $s_0$ such that $1\leq s_0\leq f$ and \begin{align*}
    &s_0\equiv3\Mod{8}\\
    &s_0\equiv -2\Mod{Apqr}.
\end{align*}
Using Lemma \ref{kronecker} and the condition that $r\equiv3$ or $5\Mod 8$, we have \begin{align*}
    \left(\frac{r}{s_0}\right)=(-1)^{\frac{r-1}{2}\frac{s_0-1}{2}}\left(\frac{s_0}{r}\right)=(-1)^\frac{r-1}{2}\left(\frac{-2}{r}\right)=-1.
\end{align*}
Applying Theorem \ref{dirichlet} with the above $a, f$, and $g=r$, there exists an odd prime $s$ such that\\
$s\equiv s_0\Mod{8Apqr}$ and $r$ is a primitive root of $s$. In other words, $r$ generates $(\Z/(s\Z))^*$ and hence there exists $0<c_0<s-1$ such that $$r^{c_0}\equiv -2\Mod{s}.$$
Since $s\equiv3\Mod{8}$, we have $\left(\frac{-2}{s}\right)=1$, implying that $c_0$ must be even.

Now by the Chinese Remainder Theorem, take any positive integer $c$ such that \begin{align*}
    &c\equiv c_0\Mod{\frac{s-1}{2}}\\
    &c\equiv 0\Mod{(p-1)(q-1)}.
\end{align*}
This is feasible because $s\equiv3\Mod{8}$ and $s\equiv -2\Mod{A}$, ensuring that $(\frac{s-1}{2},(p-1)(q-1))=1$. Since $c_0$ is even and $\frac{s-1}{2}$ is odd, this makes $c\equiv c_0\Mod{s-1}$. Thus, we obtain $r^c\equiv -2\Mod{s}$ and $r^c\equiv1\Mod{pq}$.

Finally, for any positive integers $a,b,d$ such that $(q-1)(r-1)(s-1)|a$, $(p-1)(r-1)(s-1)|b$, $d\equiv1\Mod{(p-1)(q-1)(r-1)}$, we have the following:\begin{align*}
    p^a+q^b+r^c+s^d&\equiv0+1+1-2\equiv0\Mod p\\
    p^a+q^b+r^c+s^d&\equiv1+0+1-2\equiv0\Mod q\\
    p^a+q^b+r^c+s^d&\equiv1+1+0-2\equiv0\Mod r\\
    p^a+q^b+r^c+s^d&\equiv1+1-2+0\equiv0\Mod s
\end{align*}
Therefore, for any positive integers $a,b,c,d$ as above, we have $$pqrs|p^a+q^b+r^c+s^d.$$ \qed

\printbibliography

@article{fang2018shortest,
  title={On the shortest weakly prime-additive numbers},
  author={Fang, Jin-Hui and Chen, Yong-Gao},
  journal={Journal of Number Theory},
  volume={182},
  pages={258--270},
  year={2018},
  publisher={Elsevier}
}

@article{erdos,
    author = {Erd\H{o}s, Paul and Hegyv\'ari, N},
    title = {On prime-additive numbers},
    journal = {Stud. Sci. Math. Hung.},
    volume = {27},
    pages={207--212},
    year = {1992}
}

@article{moree2008primes,
  title={On primes in arithmetic progression having a prescribed primitive root. II},
  author={Moree, Pieter},
  journal={Functiones et Approximatio Commentarii Mathematici},
  volume={39},
  number={1},
  pages={133--144},
  year={2008},
  publisher={Adam Mickiewicz University, Faculty of Mathematics and Computer Science}
}

@book{ayoub1963introduction,
    author={Ayoub, Raymond George},
    title={An introduction to the analytic theory of numbers},
    publisher = {Amer. Math. Soc.},
    year = {1963}
}

@book{hardy1979introduction,
  title={An introduction to the theory of numbers},
  author={Hardy, Godfrey Harold and Wright, Edward Maitland},
  year={1979},
  publisher={Oxford university press}
}

\end{document}